\documentclass[12pt]{amsart}
\usepackage{amssymb}
\usepackage{amsmath}
\usepackage[all]{xy}

\newtheorem{thm}{Theorem}[section]
\newtheorem{lemma}[thm]{Lemma}
\newtheorem{prop}[thm]{Proposition}

\theoremstyle{definition}

\theoremstyle{remark}
\newtheorem{rmk}[thm]{Remark}
\newtheorem{exam}[thm]{Example}

\numberwithin{equation}{section}

\newcommand{\C}{\mathbb{C}}

\newcommand{\N}{\mathbb{N}}

\newcommand{\sg}{{\mathfrak S}}
\newcommand{\x}{{\mathbf x}}
\newcommand{\ds}{\displaystyle}
\newcommand{\bfone}{\mathbf{1}}

\DeclareMathOperator{\ch}{ch}
\DeclareMathOperator{\sh}{sh}
\DeclareMathOperator{\des}{des}
\DeclareMathOperator{\Des}{DES}
\DeclareMathOperator{\Dex}{DEX}

\DeclareMathOperator{\maj}{maj}
\DeclareMathOperator{\exc}{exc}
\DeclareMathOperator{\Lie}{Lie}
\DeclareMathOperator{\Ind}{Ind}

\title[Unimodality of Eulerian quasisymmetric functions]
{Unimodality of Eulerian quasisymmetric functions}
\author[Henderson]{Anthony Henderson$^1$}
\address{School of Mathematics and Statistics\\
University of Sydney NSW 2006\\
Australia}
\email{anthony.henderson@sydney.edu.au}
\thanks{$^{1}$Supported in part by
Australian Research Council grant DP0985184.}

\author[Wachs]{Michelle L.\ Wachs$^2$}
\address{Department of Mathematics\\
University of Miami\\
Coral Gables, FL 33124\\
USA}
\email{wachs@math.miami.edu}
\thanks{$^{2}$Supported in part by National Science Foundation grant
DMS 0902323}

\subjclass[2000]{Primary 05E05; Secondary 05A15, 05A30}

\begin{document}
%%%%%%%%%%%%%%%%%%%%%%%%%%%%%%%%%%%%%%%%%%%%%%%%%%%%%%%%%%%%%
\begin{abstract}
We prove two conjectures of Shareshian and Wachs about Eulerian 
quasisymmetric functions and polynomials. The first states that the 
cycle type Eulerian quasisymmetric function $Q_{\lambda,j}$ is 
Schur-positive, and moreover that the sequence $Q_{\lambda,j}$ as $j$ 
varies is Schur-unimodal. The second conjecture, which we prove using 
the first, states that the cycle type $(q,p)$-Eulerian polynomial
\newline $A_\lambda^{\maj,\des,\exc}(q,p,q^{-1}t)$ is $t$-unimodal. 
\end{abstract}
%%%%%%%%%%%%%%%%%%%%%%%%%%%%%%%%%%%%%%%%%%%%%%%%%%%%%%%%%%%%%%%
\maketitle
%%%%%%%%%%%%%%%%%%%%%%%%%%%%%%%%%%%%%%%%%%%%%%%%%%%%%%%%%%%%%%%%%%%%%%%%%%
\section{Introduction}

The Eulerian polynomial $A_n(t) = \sum_{j=0}^{n-1} a_{n,j}t^j$ is the 
enumerator of  permutations in the symmetric group $\sg_n$ by their 
number of descents or their number of excedances.  Two well-known 
and important properties of   the Eulerian polynomials  are symmetry 
and unimodality  (see \cite[p.~292]{com}).  That is,  the sequence 
of coefficients $(a_{n,j})_{0\le j \le n-1}$  satisfies
\begin{equation}
a_{n,j} = a_{n,n-1-j}
\end{equation}
and
\begin{equation}
a_{n,0} \le a_{n,1} \le  \dots \le a_{n, \lfloor \frac{n-1}{2} 
\rfloor} = a_{n, \lfloor \frac{n}{2} \rfloor} \ge \dots \ge 
a_{n,n-2}\ge a_{n,n-1}.
\end{equation}

  Brenti \cite[Theorem 3.2]{br} showed that  the cycle type Eulerian 
polynomial  $A^{\exc}_\lambda(t)$, which enumerates permutations of 
fixed cycle type $\lambda$ by their  number of excedances, is also 
symmetric and unimodal.   More recently, Shareshian and Wachs 
\cite{sw} proved that the $q$-Eulerian polynomial
$A_n^{\maj,\exc}(q,q^{-1}t)$, which is the enumerator for the joint 
distribution of the  major index and excedance number over 
permutations in $\sg_n$, is  symmetric and unimodal when viewed as a
polynomial in $t$ with coefficients in $\N[q]$.  They showed that 
symmetry holds for the cycle type $(q,p)$-analog 
$A_{\lambda}^{\maj,\des,\exc}(q,p,q^{-1}t)$ as  a polynomial in $t$
with coefficients in $\N[q,p]$ and conjectured that unimodality holds 
as well.  (Symmetry fails for the less refined $(q,p)$-analog 
$A_{n}^{\maj,\des,\exc}(q,p,q^{-1}t)$, see \cite{sw}.)

In this paper we prove the unimodality conjecture of Shareshian and 
Wachs, by  first establishing a symmetric function analog, also 
conjectured in \cite{sw},  and then using Gessel's theory of 
quasisymmetric functions to deduce the unimodality of
$A_{\lambda}^{\maj,\des,\exc}(q,p,q^{-1}t)$.

The symmetric function analog of the unimodality conjecture  involves 
the cycle type refinements $Q_{\lambda,j}$ of the Eulerian 
quasisymmetric functions $Q_{n,j}$, which were  introduced  by 
Shareshian and Wachs \cite{sw} as a tool for  studying  the 
$q$-Eulerian polynomials  and  the $(q,p)$-Eulerian polynomials.
  Both $Q_{n,j} $ and  $Q_{\lambda,j}$ were shown to be symmetric 
functions in \cite{sw}.  Moreover  such properties as $p$-positivity, 
Schur-positivity,  and Schur-unimodality were established for 
$Q_{n,j}$ and conjectured for   $Q_{\lambda,j}$.

In subsequent work, Sagan, Shareshian and Wachs \cite{ssw} 
established  $p$-positivity of $Q_{\lambda,j}$ by proving 
\cite[Conjecture 6.5]{sw}, which gives the expansion of 
$Q_{\lambda,j}$  in the power-sum symmetric function basis. This was 
used to  obtain a cyclic sieving result for the $q$-Eulerian 
polynomials refined by cycle type.
Here we continue the study of Eulerian quasisymmetric functions by
   establishing  Schur-positivity of $Q_{\lambda,j}$ and 
Schur-unimodality of the sequence $(Q_{\lambda,j})_{0\le j\le n-1}$.

We briefly recall the main concepts involved, referring the reader to 
\cite{sw} for the background and standard notation.
The Eulerian quasisymmetric functions $Q_{n,j}$ in 
$\x=(x_1,x_2,x_3,\cdots)$, for $n,j\in\N$, are defined in \cite{sw} by
\begin{equation}
Q_{n,j}(\x):=\sum_{\substack{\sigma\in\sg_n\\\exc(\sigma)=j}}F_{\Dex(\sigma),n}(\x),
\end{equation}
where $\Dex(\sigma)$ is the subset of $[n-1]:=\{1,2,\cdots,n-1\}$ 
defined in \cite[Section 2]{sw}, and
$F_{S,n}(\x)$ is the fundamental quasisymmetric function of degree 
$n$ associated to $S\subseteq[n-1]$. It is immediate from this 
definition that $Q_{n,j}=0$ unless $j\leq n-1$.

The apparently quasisymmetric functions $Q_{n,j}$ are symmetric 
functions by \cite[Theorem 5.1(1)]{sw}, and form a symmetric 
sequence, in the sense that  $Q_{n,j}=Q_{n,n-1-j}$, by 
\cite[(5.3)]{sw}. Moreover, \cite[Theorem 1.2]{sw} shows that they 
have the following generating series:
\begin{equation} \label{qnjeqn}
\sum_{n,j}Q_{n,j}\,t^jz^n=\frac{(1-t)H(z)}{H(zt)-tH(z)},
\end{equation}
where as usual $H(z)=\sum_{n\ge 0}h_nz^n$ is the generating series of 
the complete homogeneous symmetric functions.

Symmetry and Schur-unimodality of the sequence $(Q_{n,j})_{0\leq 
j\leq n-1}$ are  consequences of \eqref{qnjeqn}.  Various ways to see 
this are given in \cite{sw}: one way involves symmetric function 
manipulations  of Stembridge \cite{stembridge}  and  another 
involves geometric considerations based on work of  Procesi \cite{pr} 
and Stanley \cite{st1}.   Indeed, \eqref{qnjeqn} implies that 
$Q_{n,j}$ is the Frobenius characteristic of the representation of 
$\sg_n$ on the degree-$2j$ cohomology of the toric variety associated 
with the Coxeter complex of $\sg_n$.  Schur-unimodality then follows 
from the hard Lefschetz theorem, see \cite{st1}.
See \cite[Section 7]{sw} for  other occurrences of $Q_{n,j}$.

The cycle type Eulerian quasisymmetric functions $Q_{\lambda,j}$, for 
$\lambda$ a partition of $n\in \N$ and $j \in \N$, are a refinement 
of the above symmetric  functions in the sense that 
$Q_{n,j}=\ds\sum_{\lambda\vdash n}Q_{\lambda,j}$. The definition in 
\cite{sw} is
\begin{equation}
Q_{\lambda,j}(\x):=\sum_{\substack{\sigma\in\sg_n\\\exc(\sigma)=j\\\lambda(\sigma)=\lambda}}F_{\Dex(\sigma),n}(\x),
\end{equation}
where $\lambda(\sigma)$ denotes the cycle type of $\sigma$. It is 
immediate from this definition that $Q_{\lambda,j}=0$ unless $j\leq 
n-k$, where $k$ is the multiplicity of $1$ as a part of $\lambda$.

The quasisymmetric functions $Q_{\lambda,j}$ are symmetric functions 
by \cite[Theorem 5.8]{sw}, and satisfy
\begin{equation}
Q_{\lambda,j}=Q_{\lambda,n-k-j}
\end{equation}
by \cite[Theorem 5.9]{sw}. These functions may all be obtained from 
those where $\lambda$ has a single part, using the operation of 
plethysm which we denote by $[\;]$. Explicitly, \cite[Corollary 
6.1]{sw} states that if $m_i$ denotes the multiplicity of $i$ as a 
part of $\lambda$, then
\begin{equation} \label{plethysm1eqn}
\sum_{j}Q_{\lambda,j}\,t^j=\prod_{i\geq 1}h_{m_i}[\sum_{j}Q_{(i),j}\,t^j].
\end{equation}
The following consequence of \eqref{plethysm1eqn} is also part of 
\cite[Corollary 6.1]{sw}:
\begin{equation} \label{plethysm2eqn}
\sum_{n,j}Q_{n,j}\,t^jz^n=\sum_{n}h_n[\sum_{i,j}Q_{(i),j}\,t^jz^i].
\end{equation}
Note that \eqref{qnjeqn},\eqref{plethysm1eqn},\eqref{plethysm2eqn} 
effectively provide an alternative definition of $Q_{\lambda,j}$. In 
this paper we will use only these equations, not the definition of 
$Q_{\lambda,j}$ in terms of quasisymmetric functions.

The first result of this paper appeared as \cite[Conjecture 5.11]{sw}.
\begin{thm} \label{schurthm}
The symmetric function $Q_{\lambda,j}$ is Schur-positive. Moreover, 
if $k$ is the multiplicity of $1$ in $\lambda$ then the symmetric 
sequence 
$$Q_{\lambda,0},Q_{\lambda,1},\cdots,Q_{\lambda,n-k-1},Q_{\lambda,n-k}$$ 
is Schur-unimodal in the sense that $Q_{\lambda,j}-Q_{\lambda,j-1}$ 
is Schur-positive for $1\leq j\leq\frac{n-k}{2}$.
\end{thm}
\noindent
The proof will be given in Section 2; it involves constructing an 
explicit $\sg_n$-representation $V_{\lambda,j}$ whose Frobenius 
characteristic is $Q_{\lambda,j}$.

We recall some basic permutation statistics.   Let $\sigma \in 
\sg_n$. The excedance number of $\sigma$ is given by
$$\exc(\sigma) :=|\{i \in [n-1] : \sigma(i) > i\}|.$$   The descent 
set of $\sigma$  is
given by $$\Des(\sigma) :=  \{i \in [n-1] : \sigma(i) > 
\sigma(i+1)\}$$ and the descent number and major index 
are$$\des(\sigma) := |\Des(\sigma) | \mbox{  and }
  \maj(\sigma):= \sum_{i\in \Des(\sigma)} i .$$

  The cycle type $(q,p)$-Eulerian polynomial is defined in \cite{sw} by
\begin{equation*}
A_\lambda^{\maj,\des,\exc}(q,p,q^{-1}t):=\sum_{\substack{\sigma\in\sg_n\\\lambda(\sigma)=\lambda}}q^{\maj(\sigma)-\exc(\sigma)}p^{\des(\sigma)}t^{\exc(\sigma)}.
\end{equation*}
This records the joint distribution of the statistics 
$(\maj,\des,\exc)$ over permutations of cycle type $\lambda$. We 
write $a_{\lambda,j}^{\maj',\des}(q,p)$ for the coefficient of $t^j$, 
which is an element of $\N[q,p]$.

The polynomial $a_{\lambda,j}^{\maj',\des}(q,p)$ may be obtained from 
the cycle type Eulerian quasisymmetric functions by a suitable 
specialization.   Explicitly, \cite[Lemma 2.4]{sw} shows that if 
$\lambda$ has the form $(\mu,1^k)$, where $\mu$ is a partition of 
$n-k$ with no parts equal to $1$, then
\begin{equation} \label{alambdaeqn}
a_{\lambda,j}^{\maj',\des}(q,p)=(p;q)_{n+1}\sum_{m\geq 0} 
p^m\sum_{i=0}^k q^{im}\,\mathbf{ps}_m(Q_{(\mu,1^{k-i}),j}),
\end{equation}
where as usual $(p;q)_i$ denotes $(1-p)(1-pq)\cdots(1-pq^{i-1})$, and 
$\mathbf{ps}_m$ is the principal specialization of order $m$. In 
\cite[Theorem 5.13]{sw}, this is used to show that 
$A_\lambda^{\maj,\des,\exc}(q,p,q^{-1}t)$ is $t$-symmetric with 
center of symmetry $\frac{n-k}{2}$, in the sense that
\begin{equation}
a_{\lambda,j}^{\maj',\des}(q,p)=a_{\lambda,n-k-j}^{\maj',\des}(q,p).
\end{equation}

The second result of this paper appeared as \cite[Conjecture 5.14]{sw}.
\begin{thm} \label{pqthm}
The $t$-symmetric polynomial 
$A_\lambda^{\maj,\des,\exc}(q,p,q^{-1}t)$ is $t$-unimodal in the 
sense that 
$$a_{\lambda,j}^{\maj',\des}(q,p)-a_{\lambda,j-1}^{\maj',\des}(q,p)\in\N[q,p]$$ 
for $1\leq j\leq\frac{n-k}{2}$, where $k$ is the multiplicity of $1$ 
in the partition $\lambda$.
\end{thm}
\noindent
The proof will be given in Section 3; it makes use of 
Theorem~\ref{schurthm} and (\ref{alambdaeqn}).
%%%%%%%%%%%%%%%%%%%%%%%%%%%%%%%%%%%%%%%%%%%%%%%%%%%%%%%%%%
\section{Proof of Theorem \ref{schurthm}}
For any positive integer $n$, we define a symmetric function $\ell_n$ by
\begin{equation}
\ell_n=\frac{1}{n}\sum_{d|n}\mu(d)p_d^{n/d},
\end{equation}
where $\mu(d)$ is the usual M\"obius function. It is well known 
\cite[Ch.~4, Proposition 4]{joyal} that $\ell_n$ is the Frobenius 
characteristic of the Lie representation $\Lie_n$ of $\sg_n$, which 
is by definition the degree-$(1,1,\cdots,1)$ multihomogeneous 
component of the free Lie algebra on $n$ generators. Here and 
subsequently, all representations and other vector spaces are over 
$\C$ (any field of characteristic $0$ would do equally well).

For us, a convenient construction of $\Lie_n$ is as the vector space 
generated by binary trees with leaf set $[n]$, subject to relations 
which correspond to the skew-symmetry and Jacobi identity of the Lie 
bracket. These relations are
\begin{equation} \label{treerel}
\begin{split}
&(T_1 \land T_2) + (T_2 \land T_1) =0\text{ and}\\
&((T_1 \land T_2) \land T_3) +((T_2 \land T_3) \land T_1) +((T_3 
\land T_1) \land T_2)= 0,
\end{split}
\end{equation}
where $A \land B$ denotes the binary tree whose left subtree is $A$ 
and right subtree is $B$, and in both cases the relation applies not 
just to the tree as a whole but to the subtree descending from any 
vertex (it being understood that the other parts of the tree are the 
same in all terms). The $\sg_n$-action is the obvious one by 
permuting the labels of the leaves. It is well known that $\Lie_n$ 
has a basis given by the trees of the form $(\cdots((s_1\land 
s_2)\land s_3)\cdots \land s_n)$ where $s_1,s_2,\cdots,s_n$ is a 
permutation of $[n]$ such that $s_1=1$.

A famous result of Cadogan \cite{cad} is that the plethystic inverse 
of $\sum_{n\geq 1} h_n$ is $\sum_{n\geq 1}(-1)^{n-1}\omega(\ell_n)$. 
A slight variant of this result is the following (compare 
\cite[Ch.~4, Proposition 1]{joyal}).
\begin{lemma} \label{cadoganlem}
We have an equality of symmetric functions:
\[
\sum_{n\geq 0} h_n[\sum_{m\geq 1}\ell_m]=(1-h_1)^{-1}.
\]
\end{lemma}
\begin{proof}
Using the well-known identity
\begin{equation}
\sum_{n\geq 0}h_n=\exp(\sum_{i\geq 1}\frac{1}{i}p_i),
\end{equation}
the left-hand side of our desired equality becomes
\[
\begin{split}
\exp(\sum_{i\geq 1}\frac{1}{i}p_i[\sum_{m\geq 1}\ell_m])
&=\exp(\sum_{i\geq 1}\frac{1}{i}p_i[\sum_{d,e\geq 1}\frac{1}{de}\mu(d)p_d^e])\\
&=\exp(\sum_{i,d,e\geq 1}\frac{1}{ide}\mu(d)p_{id}^e)\\
&=\exp(\sum_{j,e\geq 1}\sum_{d|j}\frac{1}{je}\mu(d)p_j^e)\\
&=\exp(\sum_{e\geq 1}\frac{1}{e}p_1^e)\\
&=\exp(-\log(1-p_1)),
\end{split}
\]
which equals the right-hand side.
\end{proof}

We deduce a new expression for the symmetric functions $Q_{(n),j}$.
\begin{prop} \label{basicprop}
The symmetric functions $Q_{(n),j}$ have the generating series:
\[
\sum_{n\ge 1,j\ge 0}Q_{(n),j}\,t^jz^n=h_1z+\sum_{m\geq 
1}\ell_m[\sum_{r\geq 2}(t+t^2+\cdots+t^{r-1})\,h_r z^r].
\]
\end{prop}
\begin{proof}
Let $A$ and $B$ denote the left-hand and right-hand sides of the 
equation. We know from \eqref{qnjeqn} and \eqref{plethysm2eqn} that
\[
\sum_{n\ge 0} h_n[A]=\frac{(1-t)H(z)}{H(zt)-tH(z)}.
\]
Using the well-known fact
\begin{equation}
\sum_{n\ge 0}  h_n[X+Y]=(\sum_{n\ge 0} h_n[X])(\sum_{n\ge0} h_n[Y])
\end{equation}
as well as Lemma \ref{cadoganlem}, we calculate
\[
\begin{split}
\sum_{n\ge 0} h_n[B]&=(\sum_{n\ge0} h_n[h_1z])\sum_{n\ge0} 
h_n[\sum_{m\geq 1}\ell_m[\sum_{r\geq 2}(t+t^2+\cdots+t^{r-1})\,h_r 
z^r]]\\
&= (\sum_{n\ge 0} h_n\,z^n) (\sum_{n\ge 0} h_n[\sum_{m\ge 1} \ell_m] 
)[\sum_{r\geq 2}(t+t^2+\cdots+t^{r-1})\,h_r z^r] \\
&=(\sum_{n\ge 0} h_n\,z^n)(1-\sum_{r\geq 
2}(t+t^2+\cdots+t^{r-1})\,h_r z^r)^{-1}\\
&=H(z)(1+\sum_{r\geq 1}\frac{t^r-t}{1-t}\,h_r z^r)^{-1}\\
&=H(z)\left(\frac{H(zt)-tH(z)}{1-t}\right)^{-1}\\
&=\frac{(1-t)H(z)}{H(zt)-tH(z)}.
\end{split}
\]
We conclude that $\sum_{n\ge 1} h_n [A] = \sum_{n \ge 1} h_n[B]$.  By 
applying the plethystic inverse of $ \sum_{n\ge 1} h_n$ to both sides 
of this equation we obtain $A=B$ as claimed.
\end{proof}

Proposition \ref{basicprop} allows us to construct an 
$\sg_n$-representation $V_{(n),j}$ whose Frobenius characteristic is 
$Q_{(n),j}$. We define a {\em marked set} to be a finite set $S$ such 
that $|S|\geq2$, together with an integer $j \in [|S|-1]$ called the 
\emph{mark} (cf.\ \cite{stembridge}). For $n\geq 2$, let $V_{(n),j}$ 
be the vector space generated by binary trees whose leaves are marked 
sets which form a partition of $[n]$ (when the marks are ignored) and 
whose marks add up to $j$, subject to the relations \eqref{treerel}. 
The $\sg_n$-action is by permuting the letters in the leaves.

\begin{exam}
$V_{(6),3}$ is spanned by the following trees and their 
$\sg_6$-translates, where the superscript on a leaf indicates the 
mark:
\[
\begin{split}
&\{1,2,3,4,5,6\}^{(3)},\\
&(\{1,2,3,4\}^{(2)}\land\{5,6\}^{(1)}),\\
&(\{1,2,3\}^{(2)}\land\{4,5,6\}^{(1)}),\\
&((\{1,2\}^{(1)}\land\{3,4\}^{(1)})\land\{5,6\}^{(1)}).
\end{split}
\]
The resulting expression for $V_{(6),3}$ as a representation of $\sg_6$ is
\[
\bfone\oplus\Ind_{\sg_4\times\sg_2}^{\sg_6}(\bfone)\oplus\Ind_{\sg_3\times\sg_3}^{\sg_6}(\bfone)\oplus\Ind_{\sg_2\wr\sg_3}^{\sg_6}(\Lie_3),
\]
where $\bfone$ denotes the trivial representation of a group, and 
$\Lie_3$ is regarded as a representation of the wreath product 
$\sg_2\wr\sg_3$ via the natural homomorphism to $\sg_3$.
\end{exam}

\begin{prop} \label{vnjprop}
For $n\geq 2$ and any $j$, $Q_{(n),j}=\ch V_{(n),j}$.
\end{prop}

\begin{proof}
We want to apply to Proposition \ref{basicprop} the 
representation-theoretic interpretation of plethysm given by Joyal in 
\cite{joyal}. If we take the definition of $\Lie_n$ in terms of 
binary trees and replace the set $[n]$ with an arbitrary finite set 
$I$, we obtain a vector space $\Lie(I)$. This defines a functor 
$\Lie$ from the category of finite sets, with bijections as the 
morphisms, to the category of vector spaces; such a functor is called 
an $\sg$-module (or a tensor species, in the terminology of 
\cite[Ch.~4]{joyal}). The character $\ch(\Lie)$ is by definition 
$\sum_{m\geq 1}\ch\Lie_m=\sum_{m\geq 1}\ell_m$.

We also define a graded $\sg$-module $W$ (that is, a functor from the 
category of finite sets with bijections to the category of 
$\N$-graded vector spaces) by letting $W(I)$ be the graded vector 
space with
\[
W(I)_a=\begin{cases}
\C,&\text{ if $1\leq a<|I|$,}\\
0,&\text{ otherwise,}
\end{cases}
\]
where the grading-preserving linear map $W(I)\to W(J)$ induced by a 
bijection $I\to J$ is the trivial one using only the identity map 
$\C\to\C$. The character $\ch_t(W)$, where we use the indeterminate 
$t$ to keep track of the grading in the obvious way, is clearly 
$\sum_{r\geq 2}(t+t^2+\cdots+t^{r-1})h_r$.

We can then define a graded $\sg$-module $\Lie\circ W$, the 
partitional composition of $\Lie$ and $W$, by
\[
(\Lie\circ 
W)(I):=\bigoplus_{\pi\in\Pi(I)}\Lie(\pi)\otimes\bigotimes_{J\in\pi}W(J),
\]
where $\Pi(I)$ denotes the set of partitions of the set $I$, and we 
identify a partition $\pi$ with its set of blocks. The grading on the 
tensor product of graded vector spaces is as usual, with $\Lie(\pi)$ 
considered as being homogeneous of degree zero. By \cite[Corollary 
7.6]{hend}, which is an extension of Joyal's result \cite[4.4]{joyal} 
to the graded setting, this operation of partitional composition 
corresponds to plethysm of the characters.  So we have
\begin{equation} \label{liepletheq}
\ch_t(\Lie\circ W)=\sum_{m\geq 1}\ell_m[\sum_{r\geq 
2}(t+t^2+\cdots+t^{r-1})\,h_r].
\end{equation}
Comparing this equation with Proposition 
\ref{basicprop}, we see that for $n\geq 2$, $Q_{(n),j}$ is the 
Frobenius characteristic of the representation of $\sg_n$ on the 
degree-$j$ homogeneous component of $(\Lie\circ W)[n]$. It is easy to 
see that this is equivalent to the representation $V_{(n),j}$ defined 
above.
\end{proof}

\begin{rmk} Equation~(\ref{liepletheq})  can also be obtained from 
an easy  modification of  \cite[Theorem~5.5]{w}.
\end{rmk}
\begin{rmk}
Proposition \ref{vnjprop} is analogous to Stembridge's result 
\cite[Proposition 4.1]{stembridge}, which realizes $Q_{n,j}$ as the 
Frobenius characteristic of the permutation representation of $\sg_n$ 
on what he calls codes of length $n$ and index $j$. In our 
terminology, these codes are the (possibly empty) sequences $(S_1,S_2,\cdots,S_m)$ 
of marked sets, whose underlying sets are disjoint subsets of $[n]$, 
and whose marks add up to $j$. So the marked sets appear in both 
contexts, but his result for $Q_{n,j}$ uses a representation with a 
basis consisting of sequences of marked sets, whereas our result for 
$Q_{(n),j}$ uses a representation spanned by binary trees of marked 
sets, which are subject to linear relations. The difference springs 
from the fact that the generating function \eqref{qnjeqn} for 
$Q_{n,j}$ effectively has $h_1^m$ in place of the $\ell_m$ in 
Proposition \ref{basicprop}. Since $Q_{(n),j}$ is not $h$-positive 
(see \cite[(5.4)]{sw}), it cannot be the Frobenius characteristic of 
a permutation representation.
\end{rmk}

We now define an $\sg_n$-representation $V_{\lambda,j}$ for any 
partition $\lambda\vdash n$. This is the vector space generated by 
forests $\{T_1,\cdots,T_m\}$, where each $T_i$ is either a binary 
tree whose leaves are marked sets, or a single-vertex tree whose leaf 
is a singleton set with no mark. There are further conditions: for 
each tree $T_i$, the leaves (ignoring the marks) must form a 
partition of a set $L_i$, and in turn, $L_1,\cdots,L_m$ must form a 
partition of $[n]$; the sizes $|L_1|,\cdots,|L_m|$ must be the parts 
of the partition $\lambda$, in some order; and the sum of the marks 
must be $j$. These forests are once again subject to the relations 
\eqref{treerel}. Note that if $n\geq 2$ and $\lambda=(n)$, this 
agrees with our earlier definition of $V_{(n),j}$.

\begin{exam}
$V_{(4,3,3,1),4}$ is spanned by the following three forests and their 
$\sg_{11}$-translates:
\[
\begin{split}
\{1,2,3,4\}^{(1)}\qquad &\{5,6,7\}^{(2)}\qquad \{8,9,10\}^{(1)}\qquad \{11\},\\
\{1,2,3,4\}^{(2)}\qquad &\{5,6,7\}^{(1)}\qquad \{8,9,10\}^{(1)}\qquad \{11\},\\
(\{1,2\}^{(1)}\land\{3,4\}^{(1)})\qquad &\{5,6,7\}^{(1)}\qquad 
\{8,9,10\}^{(1)}\qquad \{11\}.
\end{split}
\]
\end{exam}

\begin{prop} \label{vlambdajprop}
For any $\lambda$ and $j$, $Q_{\lambda,j}=\ch V_{\lambda,j}$.
\end{prop}

\begin{proof}
This follows by interpreting \eqref{plethysm1eqn} along the lines of 
the proof of Proposition \ref{vnjprop}, using the result of 
Proposition \ref{vnjprop} and the fact that $Q_{(1),0}=h_1$.
\end{proof}

 From this description of $Q_{\lambda,j}$, Schur-positivity is 
immediate. We can also deduce the stronger Schur-unimodality 
statement of Theorem \ref{schurthm}.
\begin{proof}[Proof of Theorem \ref{schurthm}]
For $1\leq j \leq \frac{n-k}{2}$, define a linear map 
$\phi:V_{\lambda,j-1}\to V_{\lambda,j}$ which takes a forest 
$F=\{T_1,\cdots,T_m\}$ to the sum of all forests obtained from $F$ by 
adding $1$ to the mark of one of the marked sets (for this to give an 
allowable forest, the original mark must be at most the size of its 
set minus $2$). It is clear that $\phi$ is still well-defined when 
one takes \eqref{treerel} into account, and that $\phi$ commutes with 
the action of $\sg_n$. By Proposition \ref{vlambdajprop}, in order to 
prove that $Q_{\lambda,j}-Q_{\lambda,j-1}$ is Schur-positive, we need 
only show that $\phi$ is injective (since then 
$Q_{\lambda,j}-Q_{\lambda,j-1}$ is the Frobenius characteristic of 
the cokernel of $\phi$).

Now there is some collection $\mathcal{F}$ of unmarked forests, 
depending on $\lambda$ but not on $j$, such that the marked forests 
as defined above whose underlying unmarked forest lies in 
$\mathcal{F}$ form a basis of $V_{\lambda,j}$. For example, if 
$\lambda=(n)$, we can take $\mathcal{F}$ to consist of all binary 
trees of the form $(\cdots((S_1 \land S_2) \land S_3) \cdots \land 
S_t)$ where $S_1,S_2,\cdots,S_t$ form a partition of $[n]$ and $1\in 
S_1$. Since $\phi$ only changes the marking, it is enough to prove 
the injectivity when we have fixed the underlying unmarked forest to 
be some element $F$ of $\mathcal{F}$.

We are now in a familiar situation. We have a collection of disjoint 
sets $A_1,A_2,\cdots,A_s$ (the nonsingleton leaves of  $F$) such that 
$|A_i|\geq 2$ for all $i$ and $|A_1|+|A_2|+\cdots+|A_s|=n-k$. We are 
considering a vector space $V$ with basis 
$[|A_1|-1]\times\cdots\times[|A_s|-1]$, to which we give a grading 
$V=\bigoplus_j V_j$ by the rule that $(b_1,\cdots,b_s)\in 
V_{b_1+\cdots+b_s}$ for any $b_i\in[|A_i|-1]$. We must show that for 
any $j$ such that $1\leq j\leq \frac{n-k}{2}$, the linear map 
$\phi:V_{j-1}\to V_j$ defined by
\[
\phi(b_1,\cdots,b_s)=\sum_{\substack{1\leq i\leq 
s\\b_i\leq|A_i|-2}}(b_1,\cdots,b_i+1,\cdots,b_s)
\]
is injective. This is a well-known fact, a special case of a far more 
general result on raising operators in posets \cite{pss}.
\end{proof}
%%%%%%%%%%%%%%%%%%%%%%%%%%%%%%%%%%%%%%%%%%%%%%%%%%%%%%%%%%%%%%%%%%%%%%%
\section{Proof of Theorem \ref{pqthm}}

The $p=1$ case of Theorem \ref{pqthm} follows immediately from 
Theorem~\ref{schurthm} and the observation from \cite[eq.~(2.13)]{sw} 
that $a^{\maj^\prime,\des}_{\lambda,j}(q,1) = (q;q)_n 
\mathbf{ps}(Q_{\lambda,j})$, where $\mathbf{ps}$ denotes the stable 
principal specialization (see \cite[Lemma 5.2]{sw}).  The proof for 
general $p$ makes use of the (nonstable) principal specialization as 
in (\ref{alambdaeqn}) and  is much more involved.

For permutations $\alpha \in \sg_{S}$ and $\beta \in \sg_T$ on 
disjoint sets $S,T$, let $\sh(\alpha,\beta)$ denote the set of 
shuffles of $\alpha$ and $\beta$.  That is,
$$\sh(\alpha,\beta) := \{\sigma \in \sg_{S\cup T} : \alpha \mbox{ and 
} \beta \mbox{ are subwords of $\sigma$} \}.$$
We define
$$\sh^*(\alpha,\beta):= \{\sigma \in \sh(\alpha,\beta): \sigma_1= \alpha_1\}.$$
Some care is needed with this definition in the case that $S$ is 
empty, when $\alpha$ is the empty word $\varnothing$ and $\alpha_1$ 
is not defined. We have $\sh(\varnothing,\beta)=\{\beta\}$, and we 
declare that $\sh^*(\varnothing,\beta)$ is empty unless 
$\beta=\varnothing$ also, in which case we set 
$\sh^*(\varnothing,\varnothing)=\{\varnothing\}$.

For $i,j\in\N$ with $i\leq j+1$, let $\epsilon_i^j$ denote the word 
$i,i+1,\cdots,j$ (which means the empty word $\varnothing$ if 
$i=j+1$).

\begin{lemma} \label{mainlem} Let $m,k,r\in\N$ with $r\leq k$.
Then for all $\alpha \in \sg_m$,
\begin{eqnarray*}
\sum_{i=r}^k (pq^r;q)_{i-r} & & \hspace{-.3in} \sum_{\sigma \in 
\sh(\alpha,\epsilon_{m+1}^{m+k-i} )}  (pq^i)^{\des(\sigma) +1} 
q^{\maj(\sigma)}  \\
&=& \nonumber \sum_{i=r}^k \sum_{\sigma \in 
\sh^*(\alpha,\epsilon_{m+1}^{m+k-i} )} (pq^i)^{\des(\sigma) +1} 
q^{\maj(\sigma)}.
\end{eqnarray*}
\end{lemma}

\begin{proof} We use induction on $k-r$.  The case $r=k$ is trivial, 
because both sides have only one term, namely 
$(pq^k)^{\des(\alpha)+1}q^{\maj(\alpha)}$.

Now suppose $r<k$, and that we know the result when $r$ is replaced 
by $r+1$. Then the left-hand side of our desired equation equals
\[
\begin{split}&\sum_{\sigma \in \sh(\alpha,\epsilon_{m+1}^{m+k-r} )} 
(pq^r)^{\des(\sigma) +1} q^{\maj(\sigma)} \\
&\quad + (1-pq^r)  \sum_{i=r+1}^k (pq^{r+1};q)_{i-r-1}  \sum_{\sigma 
\in \sh(\alpha,\epsilon_{m+1}^{m+k-i} )}  (pq^i)^{\des(\sigma) +1} 
q^{\maj(\sigma)} \\
&=\sum_{\sigma \in \sh(\alpha,\epsilon_{m+1}^{m+k-r} )} 
(pq^r)^{\des(\sigma) +1} q^{\maj(\sigma)} \\
&\quad + (1-pq^r)  \sum_{i=r+1}^k  \sum_{\sigma \in 
\sh^*(\alpha,\epsilon_{m+1}^{m+k-i} )}  (pq^i)^{\des(\sigma) +1} 
q^{\maj(\sigma)}.
\end{split}
\]
To complete the proof we need only show that
\begin{equation}\label{stareq}
\begin{split}
&\sum_{\tau \in \sh(\alpha,\epsilon_{m+1}^{m+k-r})\setminus 
\sh^*(\alpha,\epsilon_{m+1}^{m+k-r})}(pq^r)^{\des(\tau) +1} 
q^{\maj(\tau)}\\
&\qquad = pq^r \sum_{i=r+1}^k  \sum_{\sigma \in 
\sh^*(\alpha,\epsilon_{m+1}^{m+k-i} )}  (pq^i)^{\des(\sigma) +1} 
q^{\maj(\sigma)}.
\end{split}
\end{equation}
Now every $\tau\in\sh(\alpha,\epsilon_{m+1}^{m+k-r})\setminus 
\sh^*(\alpha,\epsilon_{m+1}^{m+k-r})$ can be written uniquely in the 
form $\epsilon_{m+1}^{m+i-r}\sigma'$ where $r<i\leq k$ and 
$\sigma'\in\sh^*(\alpha,\epsilon_{m+i-r+1}^{m+k-r})$. Subtracting 
$i-r$ from every letter of $\sigma'$ which exceeds $m$, we obtain an 
element $\sigma\in\sh^*(\alpha,\epsilon_{m+1}^{m+k-i})$. This gives a 
bijection
\[
\begin{split}
\sh(\alpha,\epsilon_{m+1}^{m+k-r})\setminus 
\sh^*(\alpha,\epsilon_{m+1}^{m+k-r})
&\leftrightarrow
\biguplus_{i=r+1}^k \sh^*(\alpha,\epsilon_{m+1}^{m+k-i})\\
\tau&\mapsto\sigma.
\end{split}
\]
It is easy to see that
\[
\begin{split}
\des(\tau) &= \des(\sigma) + 1,\\
\maj(\tau) &= \maj(\sigma) + (i-r)( \des (\sigma) +1).
\end{split}
\]
We thus have
\[
\begin{split}
&\sum_{\tau \in \sh(\alpha,\epsilon_{m+1}^{m+k-r} ) \setminus 
\sh^*(\alpha,\epsilon_{m+1}^{m+k-r} )} (pq^r)^{\des(\tau) +1} 
q^{\maj(\tau)}\\
&=\sum_{i=r+1}^k  \sum_{\sigma \in 
\sh^*(\alpha,\epsilon_{m+1}^{m+k-i} )} (pq^r)^{\des(\sigma)+2} 
q^{\maj(\sigma)+(i-r)(\des(\sigma) +1)} \\
&=pq^r \sum_{i=r+1}^k  \sum_{\sigma \in 
\sh^*(\alpha,\epsilon_{m+1}^{m+k-i} )}  (pq^i)^{\des(\sigma) +1} 
q^{\maj(\sigma)},
\end{split}
\]
which establishes \eqref{stareq}.
\end{proof}

We deduce a result about the principal specialization of order $m$.

\begin{prop} \label{fundprop} Let $k,n\in\N$ with $k\leq n$. For any 
subset $S$ of $[n-k-1]$,
\[
(p;q)_{n+1} \sum_{m\geq 0}p^m \sum_{i=0}^{k} q^{im} 
\mathbf{ps}_m(F_{S,n-k}h_{k-i})
\in\N[q,p].
\]
More precisely, this expression equals
\[
\sum_{i=0}^k \sum_{\sigma \in \sh^*(\alpha, \epsilon_{n-k+1}^{n-i})} 
(pq^i)^{\des(\sigma)+1} q^{\maj(\sigma)},
\]
where $\alpha$ is any fixed permutation in  $\sg_{n-k}$ with descent set $S$.
\end{prop}

\begin{proof}
Let $\alpha\in\sg_{n-k}$ have descent set $S$. Note that 
$h_{k-i}=F_{\emptyset,k-i}$. As a special case of the general rule 
for multiplying fundamental quasisymmetric functions (see 
\cite[Exercise 7.93]{stanley2}), we have
\begin{equation}
F_{S,n-k}h_{k-i}=\sum_{\sigma \in \sh(\alpha, 
\epsilon_{n-k+1}^{n-i})} F_{\Des(\sigma),n-i}.
\end{equation}
Hence
\[
\begin{split}
(p;q)_{n+1} \sum_{m\geq 0}p^m &\sum_{i=0}^{k} q^{im} 
\mathbf{ps}_m(F_{S,n-k}h_{k-i})\\
%&= (p;q)_{n+1} \sum_{m\geq 0}p^m \sum_{i=0}^{k} q^{im} \mathbf{ps}_m(\sum_{\sigma \in \sh
%(\alpha, \epsilon_{n-k+1}^{n-i})} F_{\Des(\sigma),n-i})
%\\
&=(p;q)_{n+1} \sum_{i=0}^k \sum_{\sigma \in \sh(\alpha, 
\epsilon_{n-k+1}^{n-i})} \sum_{m\geq 0}(pq^i)^m \mathbf{ps}_m 
(F_{\Des(\sigma),n-i})
\\
&=(p;q)_{n+1}  \sum_{i=0}^k  \sum_{\sigma \in \sh(\alpha, 
\epsilon_{n-k+1}^{n-i})} \frac{ (pq^i)^{\des(\sigma)+1} 
q^{\maj(\sigma)}}{(pq^i;q)_{n-i+1}}
\\
&=\sum_{i=0}^k (p;q)_i  \sum_{\sigma \in \sh(\alpha, 
\epsilon_{n-k+1}^{n-i})}  (pq^i)^{\des(\sigma)+1} q^{\maj(\sigma)}
\\
&=\sum_{i=0}^k \sum_{\sigma \in \sh^*(\alpha, 
\epsilon_{n-k+1}^{n-i})}  (pq^i)^{\des(\sigma)+1} q^{\maj(\sigma)},
\end{split}
\]
with the second equation following from \cite[Lemma 5.2]{gr} and the 
fourth equation following from the $r=0, m=n-k$ case of 
Lemma~\ref{mainlem}.
\end{proof}

We can now deduce Theorem \ref{pqthm}.
\begin{proof}[Proof of Theorem \ref{pqthm}]
Recall \eqref{alambdaeqn} that we can express 
$a_{\lambda,j}^{\maj',\des}(q,p)$ in terms of 
$\mathbf{ps}_m(Q_{(\mu,1^{k-i}),j})$, where $\lambda=(\mu,1^k)$ and 
$\mu\vdash n-k$ has no parts equal to $1$. It is clear from 
\eqref{plethysm1eqn} that $Q_{(\mu,1^{k-i}),j}=Q_{\mu,j}h_{k-i}$. So 
\eqref{alambdaeqn} can be rewritten
\begin{equation} \label{newalambdaeqn}
a_{\lambda,j}^{\maj',\des}(q,p)=(p;q)_{n+1}\sum_{m\geq 0} 
p^m\sum_{i=0}^k q^{im}\,\mathbf{ps}_m(Q_{\mu,j}h_{k-i}).
\end{equation}
For any $j$ such that $1\leq j\leq \frac{n-k}{2}$, we therefore have
\begin{equation} \label{differenceeqn}
\begin{split}
&a_{\lambda,j}^{\maj',\des}(q,p)-a_{\lambda,j-1}^{\maj',\des}(q,p)\\
&\qquad =(p;q)_{n+1}\sum_{m\geq 0} p^m\sum_{i=0}^k 
q^{im}\,\mathbf{ps}_m((Q_{\mu,j}-Q_{\mu,j-1})h_{k-i}).
\end{split}
\end{equation}
Now by Theorem \ref{schurthm}, $Q_{\mu,j}-Q_{\mu,j-1}$ is a 
nonnegative integer linear combination of Schur functions $s_\rho$ 
for $\rho\vdash n-k$. By \cite[Theorem 7.19.7]{stanley2}, each 
$s_\rho$ is in turn a nonnegative integer linear combination of 
fundamental quasisymmetric functions $F_{S,n-k}$ for $S\subseteq 
[n-k-1]$. So \eqref{differenceeqn} belongs to $\N[q,p]$ by 
Proposition \ref{fundprop}.
\end{proof}

%%%%%%%%%%%%%%%%%%%%%%%%%%%%%%%%%%%%%%%%%%%%%%%%%%%%%%%%%%

\section*{Acknowledgements}
The research presented here began while both authors were visiting 
the Centro di Ricerca Matematica Ennio De Giorgi at the Scuola 
Normale Superiore  in Pisa as participants in the program 
`Configuration Spaces: Geometry, Combinatorics and Topology', 
May--June 2010.  We thank the members of the scientific committee, A. 
Bj\"orner,  F. Cohen, C. De Concini, C. Procesi, and M. Salvetti, for 
inviting us to participate, and the center for its support.
%%%%%%%%%%%%%%%%%%%%%%%%%%%%%%%%%%%%%%%%%%%%%%%%%%%%%%%%%%

%%%%%%%%%%%%%%%%%%%%%%%%%%%%%%%%%%%%%%%%%%%%%%%%%%%%%%%%%%

\begin{thebibliography}{9}

\bibitem{br} F. Brenti, {\it
Permutation enumeration, symmetric functions, and unimodality},
Pacific J. Math. {\bf 157} (1993), 1--28.

\bibitem{cad} C. C. Cadogan, {\it The M\"obius function and connected graphs},
J. Combinatorial Theory Ser. B 11 (1971) 193-200.

\bibitem{com} L. Comtet, {\it Advanced Combinatorics}, Reidel, Dordrecht, 1974.

\bibitem{gr} I.~M.~Gessel and C.~Reutenauer, {\it Counting 
permutations with given cycle structure and descent set}, J.\ 
Combin.\ Theory Ser.\ A \textbf{64} (1993), no.~2, 189--215.

\bibitem{hend} A.~Henderson, {\it Representations of wreath products 
on cohomology of De~Concini--Procesi compactifications}, Int.\ Math.\ 
Res.\ Not.\ \textbf{2004} (2004), no.~20, 983--1021.

\bibitem{joyal} A.~Joyal, {\it Foncteurs analytiques et esp\`eces de 
structures}, in {\it Combinatoire Enumerative (Montreal, Quebec, 
1985)}, Lecture Notes in Math.\ \textbf{1234}, Springer, Berlin, 
1986, 126--159.

%\bibitem{klyachko} A.~A.~Klyachko, {\it Lie elements in a tensor 
%algebra}, Siberian Math.\ J.\ \textbf{15} (1974), no.~6, 914--921.

\bibitem{pr} C. Procesi,
{\it The toric variety associated to Weyl chambers},
Mots, 153--161,
Lang. Raison. Calc.,
Herm\`es, Paris, 1990.

\bibitem{pss} R.~A.~Proctor, M.~E.~Saks, and D.~G.~Sturtevant, {\it 
Product partial orders with the Sperner property}, Discrete Math.\ 
\textbf{30} (1980), no.~2, 173--180.

\bibitem{ssw} B. Sagan, J. Shareshian and M.L. Wachs, {\it Eulerian 
quasisymmetric functions and cyclic sieving}, Adv. in Applied Math., 
to appear.


\bibitem{sw} J.~Shareshian and M.~L.~Wachs, {\it Eulerian 
quasisymmetric functions}, Adv.\ Math.\ \textbf{225} (2010), no.~6, 
2921--2966.

\bibitem{st1} R.~P. Stanley,  {\it Log-concave and unimodal sequences in
algebra, combinatorics, and geometry},  Graph theory and its
applications: East and West (Jinan, 1986),  500--535, Ann. New
York Acad. Sci., 576, New York Acad. Sci., New York, 1989.


\bibitem{stanley2} R.~P.~Stanley, {\it Enumerative Combinatorics 
Volume 2}, Cambridge Studies in Advanced Mathematics \textbf{62}, 
Cambridge University Press 1999.

\bibitem{stembridge} J.~R.~Stembridge, {\it Eulerian numbers, 
tableaux, and the Betti numbers of a toric variety}, Discrete Math.\ 
\textbf{99} (1992), no.~1--3, 307--320.

\bibitem{w} M.~L.~Wachs, 
{\it Whitney homology of semipure shellable posets}, J. Algebraic 
Comb.\ \textbf{9} (1999), 173--207.

\end{thebibliography}
\end{document}